\input ./style/arxiv-vmsta.cfg
\documentclass[numbers,compress,v1.0.1]{vmsta}

\usepackage{mathbh}

\volume{2}
\issue{4}
\pubyear{2015}
\firstpage{343}
\lastpage{353}
\doi{10.15559/15-VMSTA41}

\startlocaldefs

\allowdisplaybreaks

\newcommand{\bGamma}{{\boldsymbol\varGamma}}
\newcommand{\bSigma}{{\boldsymbol\varSigma}}
\newcommand{\ba}{{\mathbf a}}
\newcommand{\bb}{{\mathbf b}}
\newcommand{\bp}{{\mathbf p}}
\newcommand{\bX}{{\mathbf X}}
\newcommand{\bL}{{\mathbf L}}
\newcommand{\bD}{{\mathbf D}}
\newcommand{\bV}{{\mathbf V}}
\newcommand{\bM}{{\mathbf M}}
\newcommand{\bY}{{\mathbf Y}}
\newcommand{\bA}{{\mathbf A}}

\newcommand{\R}{{\mathbb R}}

\newcommand{\FFF}{\mathfrak{F}}

\DeclareMathOperator{\pr}{\mathsf P}
\DeclareMathOperator{\M}{\mathsf E}
\DeclareMathOperator{\D}{Var}
\DeclareMathOperator{\cov}{Cov}
\DeclareMathOperator{\diag}{diag}
\DeclareMathOperator{\argmin}{argmin}
\newcommand{\bydef}{\stackrel{\text{\rm def}}{=}}
\newcommand{\inprob}{\stackrel{\text{\rm P}}{\longrightarrow}}
\newcommand{\weak}{\stackrel{\text{\rm W}}{\longrightarrow}}
\newtheorem{thm}{Theorem}
\newtheorem{lem}{Lemma}
\endlocaldefs

\begin{document}
\begin{frontmatter}

\title{Linear regression by observations from mixture with~varying~concentrations}
\author{\inits{D.}\fnm{Daryna}\snm{Liubashenko}}\email{dlyubashenko@gmail.com}
\author{\inits{R.}\fnm{Rostyslav}\snm{Maiboroda}\corref{cor1}}\email
{mre@univ.kiev.ua}
\cortext[cor1]{Corresponding author.}

\address{Kyiv National Taras Shevchenko University, Kyiv, Ukraine}

\markboth{D. Liubashenko, R. Maiboroda}{Linear regression by
observations from mixture with varying concentrations}

\begin{abstract}
We consider a finite mixture model with varying mixing probabilities.
Linear regression models are assumed for observed
variables with coefficients depending on the mixture component
the observed subject belongs to. A modification of the least-squares
estimator is proposed for estimation of the regression
coefficients. Consistency and asymptotic normality of the
estimates is demonstrated.
\end{abstract}


\received{20 October 2015}
\revised{22 November 2015}
\accepted{26 November 2015}
\publishedonline{4 December 2015}
\end{frontmatter}

\section{Introduction}\label{SectIntr}

In this paper, we discuss a structural linear regression technique in
the context of model of mixture with varying concentrations (MVC).
MVC means that the observed subjects belong to $M$ different
subpopulations (mixture components). The true numbers of
components to which the subjects $O_j$, $j=1,\dots,N$, belong, say,
$\kappa_j=\kappa(O_j)$,
are unknown, but we know the probabilities
\[
p_{j;N}^k\bydef\pr\{\kappa_j=k\}, \quad k=1,\dots,M
\]
(mixing probabilities or concentrations of the mixture
components).
MVC models arise naturally
in the description of medical, biologic, and sociologic data
\cite
{Autin:TestDensities,Maiboroda:StatisticsDNA,Maiboroda:AdaptMVC,Shcherbina}.
They can be considered as a generalization of finite mixture
models (FMM). Classical theory of FMMs can be found in
monographs \cite{McLachlan,Titterington}.

Let
\[
\xi(O)=\bigl(Y(O),X^1(O),\dots,X^d(O)
\bigr)^T
\]
be a vector of observed features (random variables) of a subject
$O$. We consider the following linear regression model for these
variables:
\begin{equation}
\label{EqLinRegr} Y(O)=\sum_{i=1}^d
b_i^{(\kappa(O))} X^i(O)+\varepsilon(O),
\end{equation}
where $\bb^{(m)}=(b_1^{(m)},\dots,b_d^{(m)})^T$ are nonrandom
regression coefficients for the $m$th component,
$\varepsilon(O)$ is an error term, which is assumed to be zero
mean and conditionally independent of the regressors vector
$\bX(O)=(X^1(O),\dots,X^d(O))^T$ given $\kappa(O)$.

\medskip

\noindent\textbf{Note.} We consider a subject $O$ as taken at random from
an infinite population, so it is random in this sense. The
vector of observed variables $\xi(O)$ can be considered as a
random vector even for a fixed $O$.

\medskip

Our aim is to estimate the vectors of regression coefficients
$\bb^{(k)}$, $k=1,\dots,M$, by the observations
$\varXi_N=(\xi_1,\dots,\xi_N)$, where $\xi_j=\xi(O_j)$. We assume
that $(\kappa_j,\xi_j)$ are independent for different $j$.

A statistical model similar to MVC with (\ref{EqLinRegr})
is considered in \cite{GrunLeisch},
where
a~parametric model for the conditional distributions of
$\varepsilon(O)$ given $\kappa(O)$ is assumed. For this case,
maximum likelihood estimation is proposed in \cite{GrunLeisch},
and a version of EM-algorithm is developed for numerical computation of
the estimates.

In this paper, we adopt a nonparametric approach assuming no
parametric models for $\varepsilon(O)$ and $\bX(O)$
distributions. Nonparametric and
semiparametric technique for MVC was developed in
\cite{Maiboroda2003,MS2008,Doronin}. We use the weighted empirical
moment technique to derive estimates for the regression
coefficients and then obtain conditions of consistency and
asymptotic normality of the estimates. These results are based
on general ideas of least squares \cite{Seber} and moment estimates
\cite{BorovkovMs}.

The rest of the paper is organized as follows. In Section
\ref{SectNonp}, we recall some results on nonparametric
estimation of functional moments in general MVC. The estimates
are introduced, and conditions of their consistency and
asymptotic normality are presented in Section \ref{SectAss}.
Section \ref{SectProof} contains proofs of the statements of
Section \ref{SectAss}. Results of computer simulations are
presented in Section~\ref{SectSimulat}.

\section{Nonparametric estimation for MVC}\label{SectNonp}

Let us start with some notation and definitions. We denote
by $F_m$ the distribution of $\xi(O)$ for $O$ belonging to the
$m$th component of the mixture, that is,
\[
F_m(A)\bydef\pr\bigl\{\xi(O)\in A\ |\ \kappa(O)=m\bigr\}
\]
for all measurable sets $A$. Then by the definition of MVC
\begin{equation}
\label{EqMVCdef} \pr\{\xi_j\in A\}=\sum_{k=1}^M
p_{j;N}^k F_k(A).
\end{equation}
In the asymptotic statements, we will consider the data
$\varXi_n=(\xi_1,\dots,\xi_N)$ as an element of (imaginary) series
of data $\varXi_1,\varXi_2,\dots,\varXi_N,\dots$ in which no link between
observations for different $N$ is assumed. So, in formal notation,
it should be more correct to write $\xi_{j;N}$ instead of $\xi_j$,
but we will drop the subscript $N$ when it is insignificant.

We consider an array of all concentrations for all data
sizes
\[
\bp\bydef\bigl(p_{j;N}^k, k=1,\dots,M;\ j=1,\dots,N;\ N=1,2,\dots\bigr).
\]
Its subarrays
\[
\bp_N^k\bydef\bigl(p_{1;N}^k,
\dots,p_{N;N}^k\bigr)^T\quad\mbox{and}\quad
\bp_{j;N}=\bigl(p_{j;N}^1,\dots,p_{j;N}^M
\bigr)^T
\]
are considered as vector columns,
and
\[
\bp_N=\bigl(p_{j;N}^k,\ j=1,\dots, N;\ k=1, \dots,M\bigr)
\]
as an $N\times M$-matrix. We will also consider a weight array
$\ba$ of the same structure as $\bp$ with similar notation for
its subarrays.

By the angle brackets with subscript $n$ we denote the averaging
by $j=1,\dots,n$:
\[
\bigl\langle\ba^k\bigr\rangle_N \bydef\frac{1}{N}
\sum_{j=1}^N a_{j;N}^k.
\]
Multiplication, summation, and other operations in the angle
brackets are made elementwise:
\[
\bigl\langle\ba^k \bp^m \bigr\rangle_N = \frac{1}{N} \sum_{j=1}^N a_{j;N}^k p_{j;N}^m;
\qquad
\bigl\langle\bigl(\ba^k\bigr)^2 \bigr\rangle_N = \frac{1}{N} \sum_{j=1}^N \bigl(a_{j;N}^k\bigr)^2.
\]
We define $\langle\bp^m\rangle\bydef\lim_{N\to\infty}\langle\bp
^m\rangle_N$
if this limit exists.
Let $\bGamma_N=(\langle\bp^l\bp^m\rangle_N)_{l,m=1}^M$ $={1\over N}\bp
_N^T\bp_N$
be an
$M\times M$ matrix, and $\gamma_{lm;N}$ be its $(l,m)$th minor.
The matrix $\bGamma_N$ can be considered as the Gramian matrix of vectors
$(\bp^1,\dots,\bp^M)$ in the inner product
$\langle\bp^i\bp^k\rangle_N$, so that it is nonsingular if these
vectors are linearly independent.

In what follows, $\inprob$ means convergence in probability, and
$\weak$ means weak convergence.

Assume now that model (\ref{EqMVCdef}) holds for the data
$\varXi_N$. Then the distribution $F_m$ of the $m$th component can
be estimated by the weighted empirical measure
\begin{equation}
\label{EqWEmeasure} \hat F_{m;N}(A)\bydef{1\over N}\sum_{j=1}^N a_{j;N}^m\mathbh1\{\xi_j\in A\},
\end{equation}
where\vspace*{-3pt}
\begin{equation}
\label{EqDefa} a_{j;N}^m={1\over
\det\bGamma_N}\sum
_{k=1}^M(-1)^{k+m}\gamma_{mk;N}p_{j;N}^k.
\end{equation}
It is shown in \cite{Maiboroda:StatisticsDNA} that if $\bGamma_N$
is nonsingular, then $\hat F_{m;N}$ is the minimax unbiased estimate for
$F_m$. The consistency of $\hat F_{m;N}$ is demonstrated in
\cite{Maiboroda2003} (see also \cite{Maiboroda:StatisticsDNA}).

Consider now functional moment estimation based on weighted
empirical moments.
Let $g:\R^{d+1}\to\R^k$ be a measurable function. Then to estimate\vspace*{-3pt}
\[
\bar g^{(m)}=\M\bigl[g\bigl(\xi(O)\bigr)\ \big\vert\ \kappa(O)=m\bigr]=\int
g(x)F_m(dx),\vspace*{-3pt}
\]
we
can use\vspace*{-3pt}
\[
\hat g^{(m)}_{;N}=\int g(x)\hat F_{m;N}(dx)=
{1\over
N}\sum_{j=1}^N
a_{j;N}^m g(\xi_j).\vspace*{-3pt}
\]
\begin{lem}[Consistency]\label{LemConsist}
Assume that
\begin{enumerate}
\item[\textup{1.}] $\M[|g(\xi(O))|\ |\ \kappa(O)=k]<\infty$ for all $k=1,\dots,M$.

\item[\textup{2.}] There exists $C>0$ such that $\det\bGamma_N>C$ for all $N$
large enough.
\end{enumerate}
\end{lem}

Then $\hat g^{(m)}_{;N}\inprob\bar g^{(m)}$ as $N\to\infty$.

This lemma is a simple corollary of Theorem 4.2 in
\cite{Maiboroda:StatisticsDNA}. (See also Theorem 3.1.1 in
\cite{MS2008}).
\begin{lem}[Asymptotic normality]\label{LemCLT}
Assume that
\begin{enumerate}
\item[\textup{1.}] $\M[|g(\xi(O))|^2\ |\ \kappa(O)=k]<\infty$ for all
$k=1,\dots,M$.

\item[\textup{2.}] There exists $C>0$ such that $\det\bGamma_N>C$ for all $N$
large enough.

\item[\textup{3.}] There exists the limit
\[
\bSigma=\lim_{N\to\infty}N\cov\bigl(\hat g^{(m)}_{;N}
\bigr).\vspace*{-3pt}
\]
\end{enumerate}
\end{lem}

Then
\[
\sqrt{N}\bigl(\hat g^{(m)}_{;N}-\bar g^{(m)}\bigr)
\weak N(0,\bSigma).
\]
For univariate $\hat g^{(m)}_{;N}$, the statement of the
lemma is contained in Theorem 4.2 from
\cite{Maiboroda:StatisticsDNA} (or Theorem 3.1.2 in
\cite{MS2008}). The multivariate case can be obtained from the
univariate one applying the Cram\'er--Wold device (see \cite
{Billingsley}, p.~382).

\section{Estimate for $\bb_m$ and its asymptotics}\label{SectAss}

In view of Lemma \ref{LemConsist}, we expect that, under suitable assumptions,
\begin{align*}
J_{m;N}(\bb)&\bydef{1\over N}\sum_{j=1}^N\ba_{j;N}^m\Biggl(Y_{j;N}-\sum_{i=1}^d b_iX_{j:N}^i \Biggr)^2\\
&=\int \Biggl(y-\sum_{i=1}^db_ix^i\Biggr)^2\hat F_{m;N}\bigl(dy,dx^1,\dots,dx^d\bigr)
\end{align*}
converges to
\begin{align*}
&\int \Biggl(y-\sum_{i=1}^db_ix^i\Biggr)^2 F_m\bigl(dy,dx^1,\dots,dx^d\bigr)\\
&\quad =\M \Biggl[ \Biggl(Y(O)-\sum_{i=1}^d b_iX^i(O) \Biggr)^2\ \Bigg\vert\ \kappa(O)=m \Biggr]\bydef J_{m;\infty}(\bb)
\end{align*}
as $N\to\infty$.

Since $J_{m;\infty}(\bb)$ attains its minimal value at $\bb^{(m)}$, it
is natural to suggest the argmin of $J_{m;N}(\bb)$ as an estimate
for $\bb^{(m)}$. If the weights $\ba^m$ were positive, then this argmin
would be
\begin{equation}
\label{EqEsimDef} \hat\bb_{;N}^{(m)}\bydef \bigl(
\bX^T\bA\bX\bigr)^{-1}\bX^T\bA\bY,
\end{equation}
where $\bX\bydef(X_j^i)_{j=1,\dots,N;\ i=1,\dots,d}$ is the $N\times
d$ matrix of observed regressors, $\bY\bydef(Y_1,\dots,Y_N)^T$ is
the vector of observed responses, and
$\bA\bydef\diag(a_{1;N}^m,\dots,a_{N;N}^m)$ is the diagonal
weight matrix for estimation of $m$th component. (Obviously,
$\bA$ depends on $m$, but we do not show it explicitly by a~ subscript
since the number $m$ of the component for which
$\bb_m$ is estimated will be further fixed.)

Generally speaking, by (\ref{EqDefa}) $a_{j;N}^m$ must be negative for
some $j$, so
$\hat\bb_{;N}^{(m)}$ is not necessarily an argmin of $\hat
J_{m;N}(\bb)$. But we will take $\hat\bb_{;N}^{(m)}$ as an estimate
for $\bb_m$ and call it a modified least-squares estimate for
$\bb_m$ in MVC model (MVC-LS estimate).

Let
\[
\bD^{(k)}\bydef\M \bigl[\bX(O)\bX^T(O)\ \big\vert \ \kappa(O)=k
\bigr]
\]
be the matrix of second moments of the regressors for subjects
belonging to the $k$th component. Denote the
variance of the $k$th component's  error term by
\[
\bigl(\sigma^{(k)} \bigr)^2=\M \bigl[\bigl(\varepsilon(O)
\bigr)^2\ \big\vert\ \kappa(O)=k \bigr].
\]
(Recall that
$\M [(\varepsilon(O))\ |\  \kappa(O)=k ]=0$).
In what follows, we assume
that these moments and variances exist for all components.

\begin{thm}[Consistency]\label{ThConsist}
Assume that
\begin{enumerate}
\item[\textup{1.}] $\bD^{(k)}$ and $(\sigma^{(k)})^2$ are finite for
all $k=1,\dots,M$.
\item[\textup{2.}] $\bD^{(m)}$ is nonsingular.
\item[\textup{3.}] There exists $C>0$ such that $\det\bGamma_N>C$ for all $N$
large enough.
\end{enumerate}
\end{thm}
Then $\hat\bb^{(m)}_{;N}\inprob\bb^{(m)}$ as $N\to\infty$.

\medskip

\noindent\textbf{Note.} Assumption 3 can be weakened.
Applying Theorem 4.2 from
\cite{Maiboroda:StatisticsDNA}, we can show that
$\hat\bb^{(m)}_{;N}$ is consistent if the vector $\bp^m_N$ is
asymptotically linearly independent from the vectors $\bp^i_N$,
$i\not=m$, as $N\to\infty$. To avoid complexities in this
presentation, we do not formulate the strict meaning of this
statement.

\medskip

Denote
$D^{ik(s)}\bydef\M\big[ X^i(O)X^k(O)\ \big\vert\ \kappa(O)=s\big]$,
\begin{eqnarray*}
&\bL^{ik(s)}\bydef\bigl(\M\bigl[ X^i(O)X^k(O)X^q(O)X^l(O)\ \big\vert\ \kappa(O)=s\bigr]\bigr)_{l,q=1}^d,&\\
&\bM^{ik(s,p)}\bydef\bigl(D^{il(s)}D^{kq(p)}\bigr)_{l,q=1}^d.&
\end{eqnarray*}

\begin{thm}[Asymptotic normality]\label{ThCLT}
Assume that
\begin{enumerate}
\item[\textup{1.}] $\M[(X^i(O))^4\ |\ \kappa(O)=k]<\infty$ and
$\M[(\varepsilon(O))^4\ |\ \kappa(O)=k]<\infty$
for all $k=1,\dots,M$.

\item[\textup{2.}] Matrix $\bD=\bD^{(m)}$ is nonsingular.

\item[\textup{3.}] There exists $C>0$ such that $\det\bGamma_N>C$ for all $N$
large enough.

\item[\textup{4.}] For all $s$,$p=1,\dots,M$, there exist
$\langle(\ba^m)^2\bp^s\bp^p\rangle$.
\end{enumerate}
\end{thm}

Then $\sqrt{N}(\hat\bb^{(m)}_{;N}-\bb^{(m)})\weak N(0,\bV)$,
where
\begin{equation}
\label{EqDefV} \bV\bydef\bD^{-1}\bSigma\bD^{-1}
\end{equation}
with
\begin{eqnarray}
&\displaystyle\bSigma=\bigl(\varSigma^{ik}\bigr)_{ik=1}^d,\nonumber&\\
&\displaystyle\varSigma^{ik} = \sum_{s=1}^M \big\langle\bigl(\ba^{m}\bigr)^2 \bp^s\big\rangle \bigl(D^{ik(s)}\bigl(\sigma^{(s)}\bigr)^2 + \bigl(\bb^{(s)} - \bb^{(m)}\bigr)^T \bL^{ik(s)}\bigl(\bb^{(s)} - \bb^{(m)}\bigr) \bigr)\nonumber\hspace{-6pt}&\\
&\displaystyle\qquad  - \sum_{s=1}^M \sum_{p=1}^M \big\langle\bigl(\ba^{m}\bigr)^2 \bp^s \bp^p\big\rangle \bigl(\bb^{(s)} - \bb^{(m)}\bigr)^T \bM^{ik(s,p)}\bigl(\bb^{(p)} - \bb^{(m)}\bigr).&\label{varequal2}
\end{eqnarray}

\section{Proofs}\label{SectProof}

\begin{proof}[Proof of Theorem \ref{ThConsist}.]
Note that if $\bD^{(m)}$ is nonsingular, then
\[
\bb^{(m)}=\argmin_{\bb\in\R^d}J_{m;\infty}(\bb) =\bigl(
\bD^{(m)}\bigr)^{-1}\M \bigl[\bigl(Y(O)\bX(O)\bigr)\ \big\vert\
\kappa(O)=m \bigr].
\]
By Lemma \ref{LemConsist},
\[
\bX^T\bA\bX={1\over N}\sum
_{j=1}^N a_{j;N}^m
\bX(O_j)\bX^T(O_j) \inprob\bD^{(m)}
\]
and
\[
\bX^T\bA\bY={1\over N}\sum
_{j=1}^N a_{j;N}^mY(O_j)
\bX(O_j) \inprob \M \bigl[\bigl(Y(O)\bX(O)\bigr)\ \big\vert\ \kappa(O)=m
\bigr]
\]
as $N\to\infty$. This implies the statement of the theorem.
\end{proof}

\begin{proof}[Proof of Theorem \ref{ThCLT}.]
Let us introduce a set of random vectors
$
\xi_j^{(k)}=(Y_j^{(k)},X_j^{1(k)},\break\dots,X_j^{d(k)})^T
$,
$j=1,2,\dots$,
with distributions $F_k$
that are independent for different $j$ and $k$ and independent
from $\kappa_j$. Denote
$\delta_j^{(k)}=\mathbh1\{\kappa_j=k\}$,\vspace*{-3pt}
\[
\xi_j'\bydef\sum_{k=1}^M
\delta_j^{(k)}\xi_j^{(k)}.
\]
Then the distribution of $\varXi_N'=(\xi_1',\dots,\xi_N')$ is the
same as that of $\varXi_N$. Since in this theorem we are interested in
weak convergence only, without
loss of generality, let us assume that $\varXi_N=\varXi_N'$ . By $\FFF$ we
denote the sigma-algebra
generated by $\xi_j^{(k)},\,j=1,\dots,N,\, k=1,\dots, M$.

Let us show that $\sqrt{N} (\hat\bb^{(m)}_{;N} - \bb^{(m)})$ converges
weakly to $N(0,\bV)$.
It is readily seen that\vspace*{-3pt}
\[
\sqrt{N} \bigl(\hat\bb^{(m)}_{;N} - \bb^{(m)}\bigr)
=\biggl[\frac{1}{N}\bigl(\bX^T \bA\bX\bigr)
\biggr]^{-1}\biggl[\frac{1}{\sqrt{N}}\bigl(\bX^T \bA\bY-
\bX^T \bA\bX\bb^{(m)}\bigr)\biggr].
\]

Since $\frac{1}{N}(\bX^T \bA\bX) \inprob\bD$, we need only to
show week convergence of the random vectors
$\frac{1}{\sqrt{N}}(\bX^T \bA\bY- \bX^T \bA\bX\bb^{(m)})$
to $N(0,\bSigma)$.

Denote\vspace*{-3pt}
\begin{align*}
&g(\xi_j)\bydef\Biggl(X_j^i\Biggl(Y_j - \sum_{l=1}^d X_j^l b_l^{(m)}\Biggr)\Biggr)_{i=1}^d,\\
&\gamma_j^i\bydef\frac{1}{\sqrt{N}} a_{j;N}^{(m)}X_j^i \Biggl(Y_j - \sum_{l=1}^d X_j^l b_l^{(m)}\Biggr).
\end{align*}

Obviously,
\[
\zeta_N\bydef\frac{1}{\sqrt{N}}\bigl(\bX^T \bA\bY-
\bX^T \bA\bX\bb^{(m)}\bigr)=\sqrt{N}\hat g_{;N}^{(m)}
= \Biggl(\sum_{j=1}^N
\gamma_j^i \Biggr)_{i=1}^d.
\]
We will apply Lemma \ref{LemCLT} to show that
$\sqrt{N}\hat g_{;N}^{(m)}\weak N(0,\bSigma)$.

First, let us show that $\bar g^{(m)}=\M\hat g_{;N}^{(m)}=0$.
It is equivalent to
$\M\sum_{j=1}^{N}\gamma_{j}^i = 0$ for all $i=1,\dots,d$. In fact,
\begin{align*}
&\M\sum_{j=1}^{N}\gamma_{j}^i\\
&\quad =\M \Biggl[ \M \frac{1}{\sqrt{N}}\sum_{j=1}^{N}a_{j;N}^{(m)} \sum_{s=1}^M\delta_j^{(s)} X_j^{i (s)} \Biggl(\sum_{s=1}^M \delta_j^{(s)}Y_j{(s)} -\sum_{l=1}^d \sum_{s=1}^M \delta_j^{(s)}X_j^{l (s)} b_l^{(m)}\Biggr) \Bigg\vert \FFF\Biggr]\\
&\quad =\frac{1}{\sqrt{N}} \M\sum_{j=1}^{N}a_{j;N}^{(m)}\sum_{s=1}^M p_{j;N}^s X_j^{i (s)} \Biggl(Y_j^{(s)} - \sum_{l=1}^d X_j^{l(s)}b_l^{(m)}\Biggr)\\
&\quad =\sqrt{N} \sum_{s=1}^M \underbrace{{\bigl\langle a^{(m)} \bp^s \bigr\rangle }_N}_{ \mathbh1\{m = s\}}\M\Biggl[X_1^{i (s)} \Biggl( Y_1^{(s)} -\sum_{l=1}^d X_1^{l (s)}b_l^{(m)} \Biggr)\Biggr]\\
&\quad =\sqrt{N} \M\Biggl[X_1^{i (m)} \underbrace{\Biggl(Y_1^{(m)} - \sum_{l=1}^d X_1^{l (m)} b_l^{(m)}\Biggr)}_{\varepsilon_1^{(m)}}\Biggr] = \sqrt{N} \M X_1^{i (m)} \M\varepsilon_1^{(m)} = 0.
\end{align*}

So
\[
\zeta_N= \Biggl(\sum_{j=1}^N
\gamma_j^i \Biggr)_{i=1}^d = \Biggl(
\sum_{j=1}^N \gamma_j^i-
\M\gamma_j^i \Biggr)_{i=1}^d.
\]
In view of Lemma \ref{LemCLT}, to complete
the proof, we only need  to show that
$\cov(\zeta_N)\to\bSigma$.

Denote
\[
\zeta_j^{i(s)} = \delta_j^{(s)}
X_j^{i(s)} \Biggl(\sum_{l=1}^{d}
X_j^{l(s)}\bigl(b_l^{(s)} -
b_l^{(m)}\bigr) + \varepsilon_j^{(s)}
\Biggr)
\]
and
\[
\eta_j ^{i(s)}= \delta_j^{(s)}
\Biggl(\sum_{l=1}^{d} D^{il(s)}
\bigl(b_l^{(s)} - b_l^{(m)}\bigr)
\Biggr)
\]
Then
\[
\zeta_N = \Biggl(\frac{1}{\sqrt{N}}\sum_{j=1}^N
a_j^{(m)} \sum_{s=1}^M
\bigl( \bigl(\zeta_j^{i(s)} - \eta_j^{i(s)}
\bigr) + \bigl(\eta_j^{i(s)} - \M\zeta_j^{i(s)}
\bigr)\bigr)\Biggr)_{i=1}^d = \bigl(S_1^i
+ S_2^i\bigr)_{i=1}^d,
\]
where
\[
S_1^i=\frac{1}{\sqrt{N}}\sum
_{j=1}^N a_j^{(m)} \sum
_{s=1}^M \bigl(\zeta_j^{i(s)}
- \eta_j^{i(s)}\bigr), \qquad S_2^i=
\frac{1}{\sqrt{N}}\sum_{j=1}^N
a_j^{(m)} \sum_{s=1}^M
\bigl(\eta_j^{i(s)} - \M\zeta_j^{i(s)}
\bigr).
\]
Note that
\[
\M\bigl(\zeta_j^{i(s)}-\eta_j^{i(s)}
\bigr) =\M\M\bigl[\bigl(\zeta_j^{i(s)}-\eta_j^{i(s)}
\bigr)\ \big\vert\ \kappa_j\bigr]=0.
\]
Now
\begin{align}
&\cov\bigl(S_1^i,S_2^k\bigr)\nonumber\\
&\quad =\frac{1}{N}\sum_{j=1}^N\bigl(a_{j;N}^{(m)}\bigr)^2 \sum_{s=1}^M \sum_{p=1}^M\M\bigl(\zeta_j^{i(s)} - \eta_j^{i(s)}\bigr) \eta_j^{k(p)}\nonumber\\
&\quad = \frac{1}{N}\sum_{j=1}^N\bigl(a_{j;N}^{(m)}\bigr)^2 \sum_{s=1}^M \sum_{p=1}^M\M\delta_j^{(s)} \Large\Biggl[\sum_{l=1}^d \bigl(X_j^{i(s)}X_j^{l(s)} - D^{il(s)}\bigr) \bigl(b_l^{(s)}- b_l^{(m)}\bigr)\nonumber\\
&\qquad +X_j^{i(s)} \varepsilon_j^{i(s)}\Large\Biggr] \cdot\delta_j^{(p)} \Biggl(\sum_{q=1}^d D^{kq(p)}\bigl(b_q^{(p)}- b_q^{(m)}\bigr) \Biggr)\nonumber\\
&\quad =\frac{1}{N}\sum_{j=1}^N\bigl(a_j^{(m)}\bigr)^2 \sum_{s=1}^M p_{j;N}^s \Biggl(\sum_{l=1}^d \bigl(\underbrace{D^{il(s)}- D^{il(s)}}_{0}\bigr) \bigl(b_l^{(s)} -b_l^{(m)}\bigr) + \M X_j^{i (s)}\underbrace{ \M \varepsilon_j^{(s)}}_{0}\Biggr)\nonumber\\
&\qquad \times\Biggl(\sum_{q=1}^d D^{kq(s)}\bigl(b_q^{(s)} - b_q^{(m)}\bigr) \Biggr) = 0;
\end{align}\vspace{-12pt}
\begin{align*}
\cov\bigl(S_1^i,S_1^k\bigr) &= \frac{1}{N}\sum_{j=1}^N\bigl(a_{j;N}^{(m)}\bigr)^2 \sum_{s=1}^M p_{j;N}^s \sum_{q=1}^d\sum_{l=1}^d\bigl(b_l^{(s)} - b_l^{(m)}\bigr) \M\bigl(\bigl(X_j^{i(s)}X_j^{l(s)} -D^{il(s)}\bigr)\\[3pt]
&\quad \times\bigl(X_j^{k(s)}X_j^{q(s)} - D^{kq(s)}\bigr)\bigr) \bigl(b_q^{(s)} - b_q^{(m)}\bigr)\\[3pt]
&\quad + \frac{1}{N}\sum_{j=1}^N \bigl(a_{j;N}^{(m)}\bigr)^2 \sum_{s=1}^M p_{j;N}^s D^{ik(s)}) \bigl(\sigma^{(s)}\bigr)^2;\\[3pt]
\cov\bigl(S_2^i,S_2^k\bigr) &= \frac{1}{N}\sum_{j=1}^N \bigl(a_{j;N}^{(m)}\bigr)^2 \sum_{s=1}^M p_{j;N}^s \sum_{q=1}^d \sum_{l=1}^d \bigl(b_l^{(s)} - b_l^{(m)}\bigr)D^{il(s)}\\[2pt]
&\quad \times\Biggl[ \bigl(b_q^{(s)} - b_q^{(m)}\bigr) D^{kq(s)} - \sum_{r=1}^M p_{j;N}^r \bigl(b_q^{(r)} - b_q^{(m)}\bigr) D^{kq(r)}\Biggr].
\end{align*}

Thus,
\begin{align*}
\cov\zeta_N &= \frac{1}{N}\sum_{j=1}^N \bigl(a_{j;N}^{(m)}\bigr)^2 \sum_{s=1}^M p_{j;N}^s D^{ik(s)} \bigl(\sigma^{(s)}\bigr)^2\\[2pt]
&\quad + \frac{1}{N}\sum_{j=1}^N\bigl(a_j^{(m)}\bigr)^2 \sum_{s=1}^M p_{j;N}^s \sum_{q=1}^d \sum_{l=1}^d\bigl(b_l^{(s)} - b_l^{(m)}\bigr)\Biggl[ \bigl(b_q^{(s)} - b_q^{(m)}\bigr)\\[2pt]
&\quad \times \underbrace{\M X_j^{i(s)} X_j^{k(s)}X_j^{l(s)}X_j^{q(s)}}_{L_{lq}^{ik(s)}} - D^{il(s)} \sum_{r=1}^M p_{j;N}^r \bigl(b_q^{(r)} - b_q^{(m)}\bigr) D^{kq(r)}\Biggr]\Bigg)_{i,k=1}^d.
\end{align*}

From the last equation we get $\cov(\zeta_N)\to\bSigma$ as
$N\to\infty$.
\end{proof}

\section{Results of simulation}\label{SectSimulat}

To assess the accuracy of the asymptotic results from Section
\ref{SectAss}, we performed a small simulation study. We considered a
two-component mixture ($M=2$) with mixing
probabilities $p_{j;N}^1=j/N$ and $p_{j;N}^2=1-p_{j;N}^1$. For
each subject, there were two observed variables $X$ and $Y$, which
were simulated based on the simple\vadjust{\eject} linear regression model
\[
Y_j=b_0^{\kappa_j}+b_1^{\kappa_j}X_j^{(\kappa_j)}+
\varepsilon_j^{(\kappa_j)},
\]
where $\kappa_i$ is the number of component the $j$th
observation belongs to, $X_j^{(1)}$ was simulated as $N(1,1)$,
$X_j^{(2)}$ as $N(2,2.25)$, and $\varepsilon_j^{(k)}$ were zero-mean Gaussians
with standard deviations 0.01 for the first component and 0.05
for the second one. The values of the regression coefficients
were
$b_0^1=3$, $b_1^1=0.5$, $b_0^2=-2$, $b_1^2=1$.

The means and covariances of the estimates were calculated over 2000
replications.

The results of simulation are presented in Table \ref{Table1}.
The true values of parameters and asymptotic covariances are placed
in the last rows of the tables.
\begin{table}
\textwidth=8.7cm
\caption{Simulation results}\label{Table1}
\begin{tabular*}{\textwidth}{llllll}
\hline
\multicolumn{6}{l}{First component}\\
\hline\\[-8pt]
$n$ & $\M\hat b_0^1$ & $\M\hat b_1^1$ & $N\D\hat b_0^1$ & $N\D\hat b_1^1$ & $N\cov(\hat b_0^1,\hat b_1^1)$\\[1pt]
\hline
500 & 3.0014 & 0.5235 & 47.61 & 45.83 & $-$ 42.27 \\
1000 & 3.0033 & 0.512 & 41.19 & 37.50 & $-$35.04 \\
2000 & 3.0011 & 0.5032& 38.84 & 34.71 & $-$32.87 \\
5000 & 3.0003 & 0.5016& 39.54 & 34.49 & $-$32.83 \\[3pt]
$\infty$ & 3 & 0.5 & 39.13 & 33.96 &    $-$32.53\\

%
\hline
\multicolumn{6}{l}{Second component}\\
\hline\\[-8pt]
$n$ & $\M\hat b_0^2$ & $\M\hat b_1^2$ & $N\D\hat b_0^2$ & $N\D\hat b_1^2$ & $N\cov(\hat b_0^2,\hat b_1^2)$\\[1pt]
\hline
500  & $-$2.0243 & 1.0084 & 67.37 & 7.94 & $-$22.17 \\
1000 & $-$42.0100 & 1.0027 & 63.04 & 7.57 & $-$20.90 \\
2000 & $-$2.0039 & 1.0016 & 63.57 & 7.52 & $-$20.95 \\
5000 & $-$2.0074 & 1.0025 & 62.41 & 7.32 & $-$20.48 \\[3pt]
$\infty$ & $-$2 & 1 & 62.20 & 7.34 & $-$20.47\\
\hline
\end{tabular*}
\end{table}

The presented data show good concordance with the asymptotic
theory for $n> 1000$.

\section{Conclusions}

We considered a modification of least-squares estimators for
linear regression coefficients in the case where observations are
obtained from a mixture with varying concentrations. Conditions of
consistency and asymptotic normality of the estimators were
derived, and dispersion matrices were evaluated.
The results of simulations confirm good concordance of estimators
covariances with the asymptotic formulas for sample sizes larger
then 1000 observations.

In real-life data analysis, concentrations (mixing probabilities) are
usually not
known exactly but estimated. So, to apply the proposed technique,
we also need to analyze sensitivity of the estimates to
perturbations of the concentrations model. (We are thankful
to the unknown referee for this observation). It is worth
noting that performance of these estimates will be poor if the
true concentrations of the components are nearly linearly
dependent ($\det\bGamma_N\approx0$). We also expect stability of
the estimates w.r.t.\ concentration perturbations if
$\det\bGamma_ N$ is bounded away from zero. More deep analysis of
sensitivity will be a~part of our further work.


\end{document}